\documentclass[12pt]{amsart}
\usepackage{amsfonts, amssymb, latexsym}
\usepackage{tikz}
\usepackage{tikz-cd}
\usepackage{hyperref}
\usepackage{wrapfig}
\usepackage{enumerate}

\newtheorem{theorem}{Theorem}[section]
\newtheorem*{theorem*}{Theorem}
\newtheorem{proposition}[theorem]{Proposition}

\newtheorem*{claim*}{Claim}

\newtheorem{Main Conjecture}[theorem]{Main Conjecture}

\theoremstyle{remark}

\newtheorem{example}[theorem]{Example}

\theoremstyle{plain}

\newcommand{\excise}[1]{}

\newcommand{\bm}[1]{\mbox{\boldmath $#1$}}

\begin{document}
\pagestyle{plain}

\mbox{}
\title{Rhombic tilings and Bott-Samelson varieties}
\author{Laura Escobar}
\address{Department of Mathematics \\ U.~Illinois at Urbana--Champaign \\ Urbana, IL 61801 \\ USA}
\email{lescobar@illinois.edu}
\author{Oliver Pechenik}
\address{Department of Mathematics \\ U.~Illinois at Urbana--Champaign \\ Urbana, IL 61801 \\ USA}
\email{pecheni2@illinois.edu}
\author{Bridget Eileen Tenner}
\address{Department of Mathematical Sciences \\ DePaul University \\ Chicago, IL 60614 \\ USA}
\email{bridget@math.depaul.edu}
\author{Alexander Yong}
\address{Department of Mathematics \\ U.~Illinois at Urbana--Champaign \\ Urbana, IL 61801 \\ USA}
\email{ayong@uiuc.edu}
\date{June 9, 2016}

\begin{abstract}
S.~Elnitsky (1997) gave an elegant bijection between rhombic tilings of $2n$-gons and commutation classes of reduced words in the symmetric group on $n$ letters. P.~Magyar (1998) found an important construction of the Bott-Samelson varieties introduced by H.C.~Hansen (1973) and M.~Demazure (1974). We explain a natural connection between S.~Elnitsky's and P.~Magyar's results. This suggests using tilings to encapsulate Bott-Samelson data (in type $A$). It also indicates a geometric perspective on S.~Elnitsky's combinatorics. We also extend this construction by
assigning desingularizations to the 
zonotopal tilings considered by B.~Tenner (2006).
\end{abstract}

\maketitle

\section{Introduction}
Let ${X}={\sf Flags}({\mathbb C}^n)$ be the variety of
complete flags 
${\mathbb C}^0\subset F_1\subset F_2\subset
\cdots\subset F_{n-1}\subset {\mathbb C}^n$. 
The group ${\sf GL_n}(\mathbb{C})$
acts on the variety $X$ by change of basis, as does
its subgroup ${\sf B}$ of invertible upper triangular matrices and its maximal torus ${\sf T}$ of 
invertible diagonal matrices.
The ${\sf T}$-fixed points are
in bijection with permutations $w$ in the symmetric group ${\mathfrak S}_n$: they are the flags 
$F_{\bullet}^{(w)}$ defined by $F_k^{(w)}=\langle \vec e_{w(1)}, \vec e_{w(2)},\ldots, \vec e_{w(k)}\rangle$ where $\vec e_i$ is the $i$-th standard basis vector. 
The {\bf Schubert variety} $X_w$ is the ${\sf B}$-orbit closure of $F_{\bullet}^{(w)}$.

There is longstanding interest in 
singularities of Schubert varieties; see, for example, the text by S.~Billey-V.~Lakshmibai \cite{Billey-Lakshmibai}. Famously, H.C.~Hansen \cite{Hansen} and M.~Demazure \cite{Demazure} independently presented (in all Lie types) resolutions of singularities $BS^{(i_1,i_2,\ldots,i_{\ell(w)})}$ of $X_w$, one for each reduced word 
$s_{i_1}s_{i_2}\cdots s_{i_{\ell(w)}}$ of $w$. M.~Demazure called these resolutions 
{\bf Bott-Samelson varieties} in reference to 
a related construction of R.~Bott-H.~Samelson \cite{Bott.Samelson}. 
In more recent work, P.~Magyar \cite{Magyar} found an important description of Bott-Samelson varieties. 

We propose a canonical connection between
P.~Magyar's work and the rhombic tilings of S.~Elnitsky \cite{Elnitsky}. In this way, tilings graphically
encapsulate Bott-Samelson data. (One should compare what follows to the similar use of X.~Viennot's {\it heaps}
\cite{Viennot}
to present Bott-Samelsons; see N.~Perrin's \cite{Perrin} and B.~Jones-A.~Woo's \cite{Woo}.)

\begin{figure}[htbp]
\begin{center}
\includegraphics[scale=0.35]{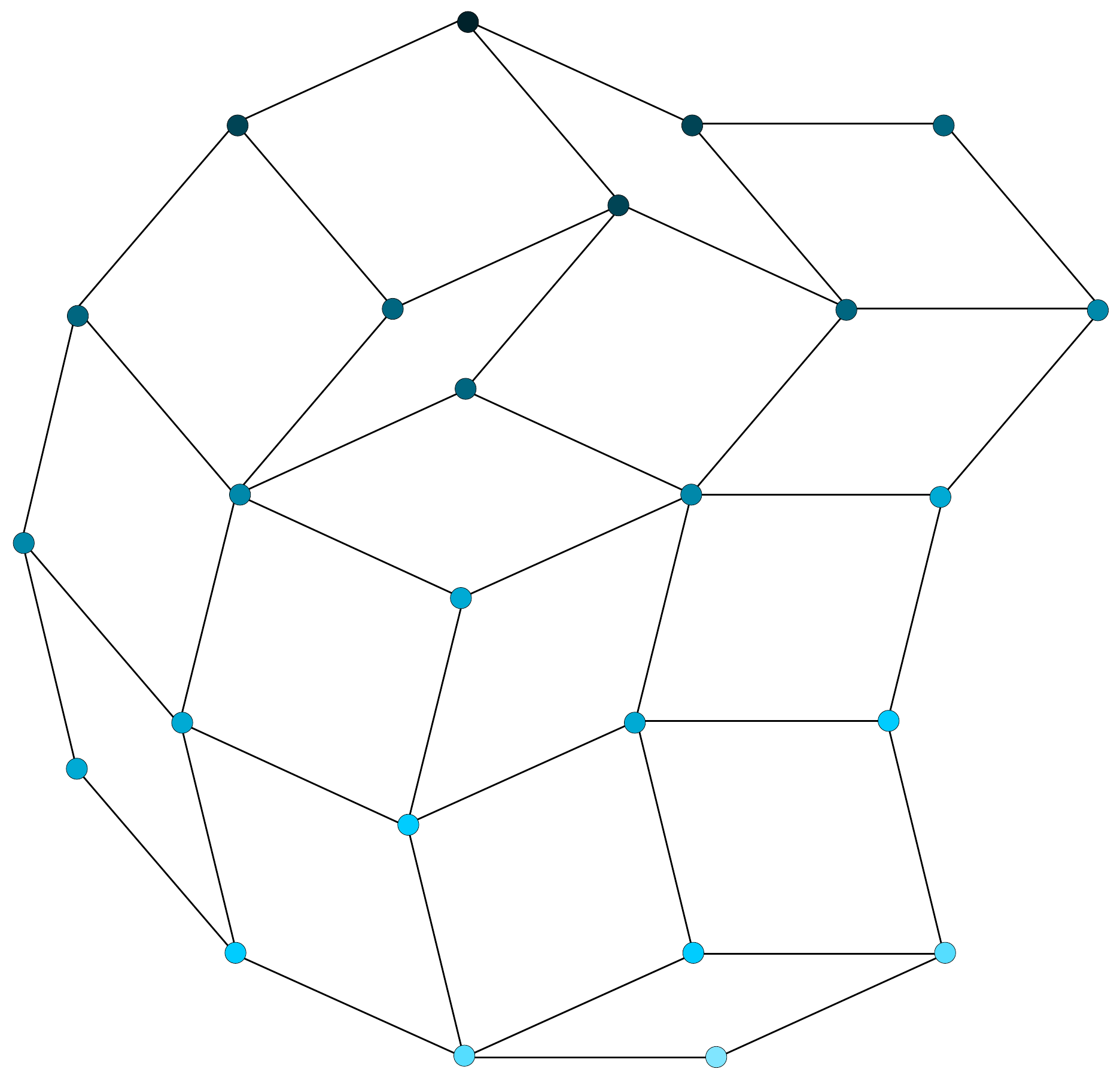}
\put(-76,-10){$\mathbb{C}^0$}
\put(-118,-10){$\mathbb{C}^1$}
\put(-160,8){$\mathbb{C}^ 2$}
\put(-192,42){$\mathbb{C}^ 3$}
\put(-201,87){$\mathbb{C}^ 4$}
\put(-190,128){$\mathbb{C}^ 5$}
\put(-162,162){$\mathbb{C}^ 6$}
\put(-112,184){$\mathbb{C}^ 7$}
\put(-27,17){$G_1$}
\put(-35,56){$G_2$}
\put(-25,93){$G_3$}
\put(-1,127){$G_4$}
\put(-78,167){$G_6$}
\put(-38,167){$G_5$}
\end{center}
\caption{The rhombic tiling picture of Bott-Samelson varieties, for the polygon ${\sf E}(7456312)$.}\label{fig:rhombic tiling example}
\end{figure}
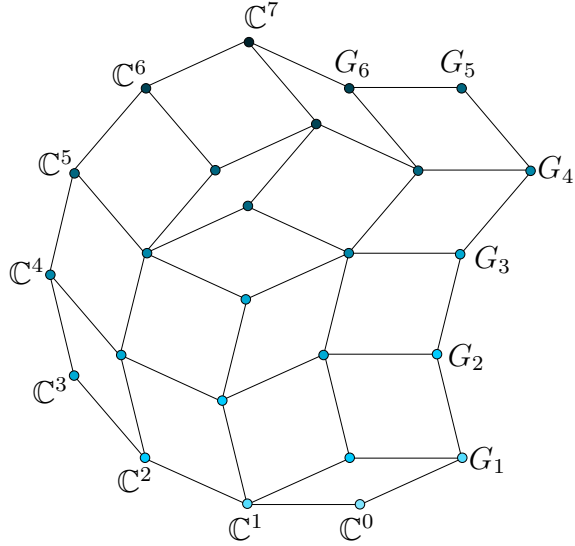

Given a permutation $w \in S_n$, the {\bf Elnitsky $\bm{2n}$-gon} ${\sf E}(w)$
has sides of length one, and these are labeled, in order, by
$1,2,\ldots,n, w(n), w(n-1), \ldots, w(1),$
in which the first $n$ labels form half of a regular $2n$-gon, and sides with the same label are parallel.  In Figure~\ref{fig:rhombic tiling example}, we give the Elnitsky $14$-gon for the permutation $7456312 \in S_7$; this example will be referenced throughout this work.

Let ${\mathcal T}(w)$ be the set of {\bf rhombic tilings} of ${\sf E}(w)$ in which the rhombi have sides of length one and edges parallel to edges of ${\sf E}(w)$. The main result of S.~Elnitsky's aforementioned work is that the set ${\mathcal T}(w)$ is in bijection with the commutation classes of reduced words of $w$ \cite[Theorem 2.2]{Elnitsky}.

We associate a vector space to each vertex of a tiling $T \in {\mathcal T}(w)$. Starting with the vertex between the edges labeled $1$ and $w(1)$, label the vertices of ${\sf E}(w)$ in clockwise order by
$${\mathbb C}^0, {\mathbb C}^1,\ldots, {\mathbb C}^n, G_{n-1}, G_{n-2}, \ldots, G_1.$$
In general,
let $V_{x}$ be the vector space associated to a vertex $x$ in the tiling.  The dimension of 
$V_x$ is the minimal path length from ${\mathbb C}^0$ to $x$ along tile edges. In Figure~\ref{fig:rhombic tiling example}, we have only labeled the external vertices.

For adjacent vertices $x$ and $y$ in a tiling $T\in {\mathcal T}(w)$, write 
$x\to y$
if $\dim (V_x)+1=\dim(V_y)$. Let
${\tt Vert}(T)$
be the vertices of $T$, and define 
\[{\mathcal Z}_{T}:=\big\{(V_x: x\in {\tt Vert}(T)): V_y\subseteq V_z \text{\ \ if $y\to z$}\big\}\subset
\prod_{x\in {\tt Vert}(T)} {\sf Gr}_{\dim(V_x)}({\mathbb C}^n),
\]
where ${\sf Gr}_k({\mathbb C}^n)$ is the Grassmannian of $k$-dimensional subspaces of ${\mathbb C}^n$.

Define the map
$\pi:{{\mathcal Z}}_{T  }\to X$
by forgetting all vector spaces except those labeled by the vertices $G_1,G_2,\ldots,G_{n-1}$. In our example, $\pi$
maps the point depicted in Figure~\ref{fig:rhombic tiling example} to the complete
flag 
\[\mathbb{C}^0\subset G_1\subset G_2\subset G_3\subset G_4\subset G_5\subset G_6\subset {\mathbb C}^7.\] 

The following theorem suggests a Schubert-geometric interpretation of tilings of Elnitsky
polygons.

\begin{theorem}
\label{thm:main}
For $T\in {\mathcal T}(w)$, ${\mathcal Z}_{T}$ is a Bott-Samelson variety, i.e.,
  a desingularization $\pi:{\mathcal Z}_{T}\to X_w$. Conversely,
every Bott-Samelson variety $BS^{(i_1,\ldots,i_{\ell(w)})}$ is canonically isomorphic to ${\mathcal Z}_{T}$ for some $T\in\mathcal{T}(w)$ where 
$w=s_{i_1}\ldots s_{i_\ell(w)}$ and $T$ is given 
in an explicit manner by \cite[Theorem~2.2]{Elnitsky}.
\end{theorem}

In Section~\ref{sec:proof}, we prove Theorem~\ref{thm:main}. The remainder of this paper concerns other Bott-Samelson data encoded by tilings.  In Section~\ref{sec:flip_zone}, we explain how the hexagon \emph{flips} of \cite[Section~3]{Elnitsky} may be interpreted geometrically. This naturally leads to the \emph{zonotopal tilings} of \cite{Tenner-RDPP}, each of which corresponds to a desingularization of a Schubert variety. We collect some additional discussion in Section~\ref{sec:final}; in particular, we explain how coloring rhombi of a tiling 
describes ${\sf T}$-fixed points as well as 
a standard stratification of a Bott-Samelson variety.

\section{Proof of Theorem~\ref{thm:main}}\label{sec:proof}

Two reduced words for $w$ are {\bf commutation equivalent} if they can be obtained from one another using only the relation
$s_is_j=s_js_i$ when $|i-j|>1$. 

By \cite[Theorem~2.2]{Elnitsky}, the set ${\mathcal T}(w)$ bijects with commutation classes  of reduced words of $w$. (Note that our orientation of the polygon is a horizontal reflection of the orientation given in \cite{Elnitsky}.) To link with \cite{Magyar}, we recall the bijection. Consider a tiling $T\in {\mathcal T}(w)$. The edges of $T$ that coincide with edges of ${\sf E}(w)$ inherit the labels of those edges, and we label the interior edges of $T$ so that parallel edges have the same labels.

Let $B_0$ be the {\bf base boundary} of ${\sf E}(w)$, formed by the edges of the polygon appearing clockwise between ${\mathbb C}^0$ and ${\mathbb C}^n$. Pick any rhombus $R_1$ of $T$ that shares two edges with $B_0$. Set $i_1 := d_1+1$, where $d_1$ is the minimum distance from ${\mathbb C}^0$ to $R_1$. Remove $R_1$ and define a new boundary, $B_1$, from $B_0$ by using the other two edges of $R_1$ instead. Now repeat this process: pick any rhombus $R_2$ that shares two edges with $B_1$; set $i_2 := d_2 + 1$, where $d_2$ is the minimum distance from ${\mathbb C}^0$ to $R_2$; remove $R_2$ and form a new boundary $B_2$. Iterating this process an additional $\ell(w) - 2$ times produces $(i_1,i_2,\ldots,i_{\ell(w)})$, for which $s_{i_1}s_{i_2} \cdots s_{i_{\ell(w)}}$ represents a commutation class of reduced words for $w$. The other direction of the bijection is indicated below.

We now show that ${\mathcal Z}_{T}$ is isomorphic to  $BS^{(i_1,i_2,\ldots,i_{\ell(w)})}$. 
P.~Magyar \cite[Theorem~1]{Magyar} describes $BS^{(i_1,i_2,\ldots,i_{\ell(w)})}$ as a list $(F_\bullet^0,\ldots,F_\bullet^m)$ of $m+1$ flags where $F_\bullet^0$ 
is the base 
flag, and such that $F_\bullet^k$ agrees with $F_\bullet^{k-1}$ everywhere except possibly on the $i_k$-th subspace. Such a
list of flags transparently corresponds in a one-to-one fashion to points in ${\mathcal Z}_{T}$: $F_\bullet^0$ is the base flag which is on the 
base boundary $B_0$ and in general, $F_\bullet^k$ is the flag on $B_k$. 

 Suppose that ${\bf j}=(j_1,j_2,\dots)$ is commutation equivalent to ${\bf i }=(i_1,i_2,\dots)$. It is well-known to experts that $BS^{\bf i}$ and $BS^{\bf j}$ are isomorphic varieties, but we include a proof for completeness. It suffices
to prove this when
${\bf j} = (i_1,\dots,i_{k+1},i_{k},\ldots, i_{\ell(w)})
$ 
differs from ${\bf i}$ only in positions $k$ and $k+1$. The general result then follows by induction.
Now, $(F_\bullet^0,\ldots,F_\bullet^m)$ is equivalent to a list of subspaces $(V_1,V_{2},\ldots)$ satisfying:
\begin{itemize}
\item $\dim(V_k)=i_k$;
\item $\mathbb{C}^{i_1-1}\subset V_{1}\subset \mathbb{C}^{i_1+1}$; that is, $V_{1}$ is contained in the $(i_1+1)$-dimensional subspace of $F_\bullet^0$ and contains the  $(i_1-1)$-dimensional subspace of $F_\bullet^0$;
\item $V_{2}$ is contained in the $(i_2+1)$-dimensional subspace of $F_\bullet^1$ and contains the  $(i_2-1)$-dimensional subspace of $F_\bullet^1$; and so on.
\end{itemize}
Since $|i_{k+1}-i_k|>1$, the $(i_k+1)$-, $(i_k-1)$-, $(i_{k+1}+1)$-, and $(i_{k+1}-1)$-dimensional subspaces of $F_\bullet^k$ are precisely the subspaces of $F_\bullet^{k-1}$ with those dimensions. So if a generic element of $BS^{\bf i}$ is $(V_{1},V_{2},\ldots)$, then a generic element of $BS^{\bf j}$ is $(V_{1},V_{2},\ldots,V_{{k+1}},V_{{k}},\ldots)$. 
That is, the isomorphism by switching factors: 
\[\tau_k :{\sf Gr}_{i_1}({\mathbb C}^n)\times\cdots\times {\sf Gr}_{i_k}({\mathbb C}^n) 
\times {\sf Gr}_{i_{k+1}}({\mathbb C}^n) \times \cdots
\to {\sf Gr}_{i_1}({\mathbb C}^n)\times\cdots\times {\sf Gr}_{i_{k+1}}({\mathbb C}^n)
\times {\sf Gr}_{i_{k}}({\mathbb C}^n)\times \cdots\]
restricts to a canonical isomorphism from 
$BS^{(i_1,i_2,\ldots)}$ to $BS^{(i_1,\dots,i_{k+1},i_{k},\ldots)}$. In other words,
${\mathcal T}(w)$ indexes Bott-Samelson varieties up to 
commutation equivalence.

Given ${\bf i} = (i_1,i_2,\ldots)$ representing a commutation class for $w$ (that is, $s_{i_1}s_{i_2}\cdots$ is a reduced decomposition of $w$), the inverse map to S.~Elnitsky's bijection constructs an ordered tiling of ${\sf E}(w)$, as follows. For $k \ge 1$, set $w^{(k)} := s_{i_1}s_{i_2}\cdots s_{i_k}$. By \cite{Elnitsky}, for $1 \le k \le \ell(w)$, the values $w^{(k)}(i_k)$ and $w^{(k)}(i_k+1)$ label adjacent edges of the boundary $B_{k-1}$. Place a rhombus, $R_k$, so that two of its edges coincide with the edges labeled $w^{(k)}(i_k)$ and $w^{(k)}(i_k+1)$ in $B_{k-1}$, and define the new boundary $B_k$ from $B_{k-1}$ by using the other two edges of $R_k$. This explicitly picks ${\mathcal Z}_{T}$ from ${\bf i}$ such that ${\mathcal Z}_{T}\cong 
BS^{\bf i}$, as desired.\qed
\begin{example}
Consider the tiling $T \in \mathcal{T}(7456312)$ depicted in Figure~\ref{fig:rhombic tiling example}. One way to select the rhombi $\{R_1, R_2,\ldots\}$ described in the proof of Theorem~\ref{thm:main} is shown in Figure~\ref{fig:rhombic tiling example with labels}, where we have recorded only the subscript $k$ of the rhombus $R_k$. The labeling in this figure represents the commutation class of the reduced word
\[s_3s_4s_2s_5s_6s_5s_3s_4s_3s_2s_1s_5s_2s_3s_6s_4s_5\]
for the permutation $7456312$. Any other such labeling of these tiles would produce a different, but commutation equivalent, reduced word. For example, the labeling obtained by swapping the selections for  $R_{14}$ and $R_{15}$, both of which share two edges with the boundary $B_{15}$, as indicated in Figure~\ref{fig:rhombic tiling example with labels},  produces the commutation equivalent reduced word
\[s_3s_4s_2s_5s_6s_5s_3s_4s_3s_2s_1s_5s_2s_6s_3s_4s_5.\]
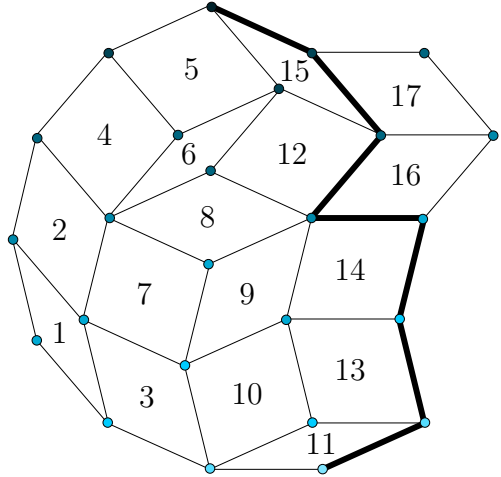
\begin{figure}[htbp]
\begin{center}
\includegraphics[scale=0.35]{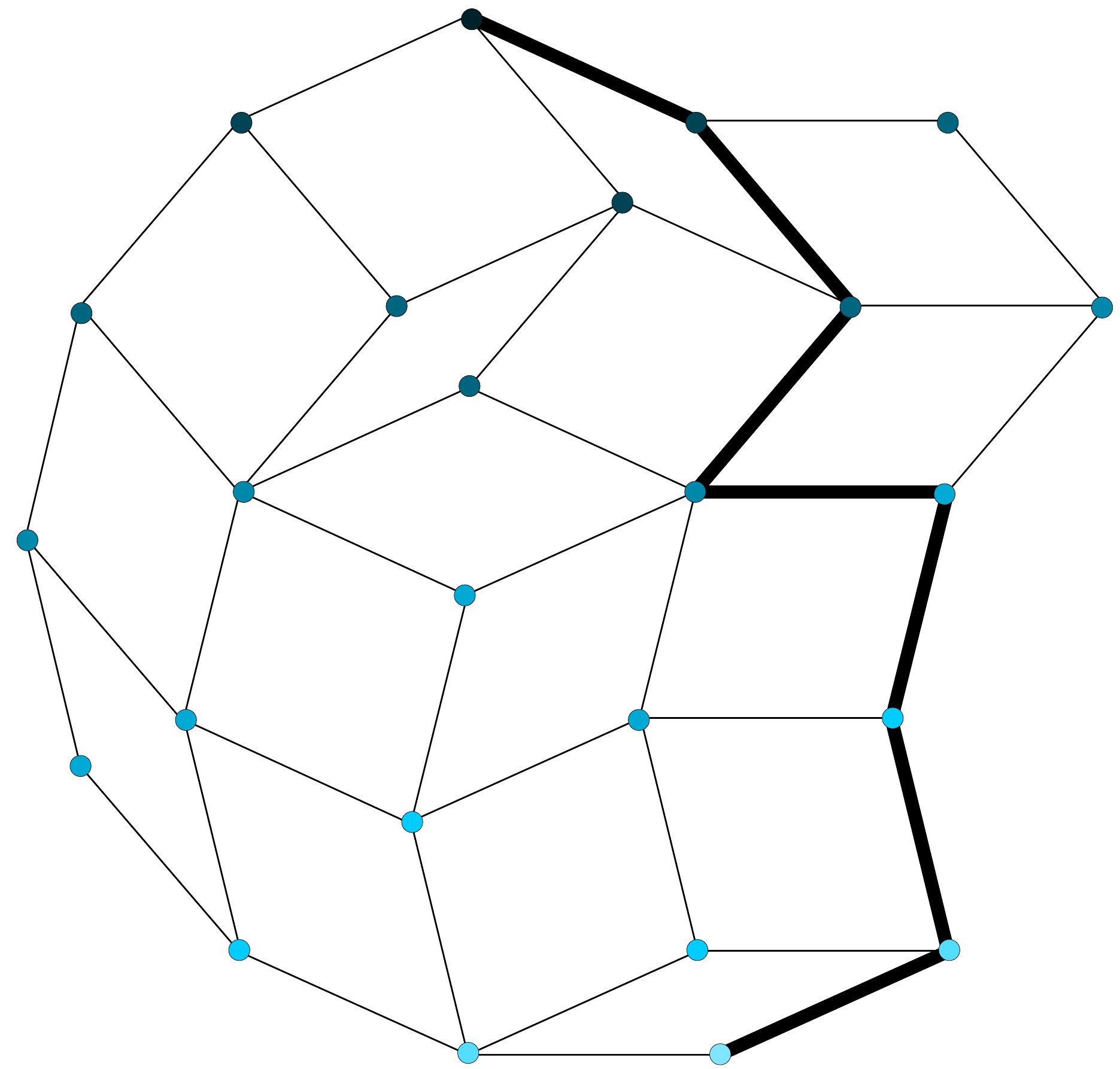}
\put(-42,140) {$17$}
\put(-42,110) {$16$}
\put(-84,149.5) {$15$}
\put(-63,74) {$14$}
\put(-63,36) {$13$}
\put(-85,118) {$12$}
\put(-74,8) {$11$}
\put(-102,27) {$10$}
\put(-99,65) {$9$}
\put(-114,94) {$8$}
\put(-138,65) {$7$}
\put(-121,118) {$6$}
\put(-120,150) {$5$}
\put(-153,125) {$4$}
\put(-137,26) {$3$}
\put(-170,89) {$2$}
\put(-170,50) {$1$}
\end{center}
\caption{A labeling of the rhombi in an element of ${\mathcal T}(7456312)$, corresponding to the reduced word $s_3s_4s_2s_5s_6s_5s_3s_4s_3s_2s_1s_5s_2s_3s_6s_4s_5$  for the permutation $7456312$. The boundary $B_{15}$ is indicated by thick line segments.}\label{fig:rhombic tiling example with labels}
\end{figure}
\end{example}

\section{Flips and zonotopal tilings}\label{sec:flip_zone}

\subsection{Flips}\label{sec:flip}
Any pair of rhombic tilings of ${\sf E}(w)$ are connected by a sequence of hexagon ``flips'' \cite[Section~3]{Elnitsky}. 
The effect of a single flip is depicted in Figure~\ref{fig:hex flip}. 

\begin{figure}[htbp]
\begin{center}
\includegraphics[scale=0.4]{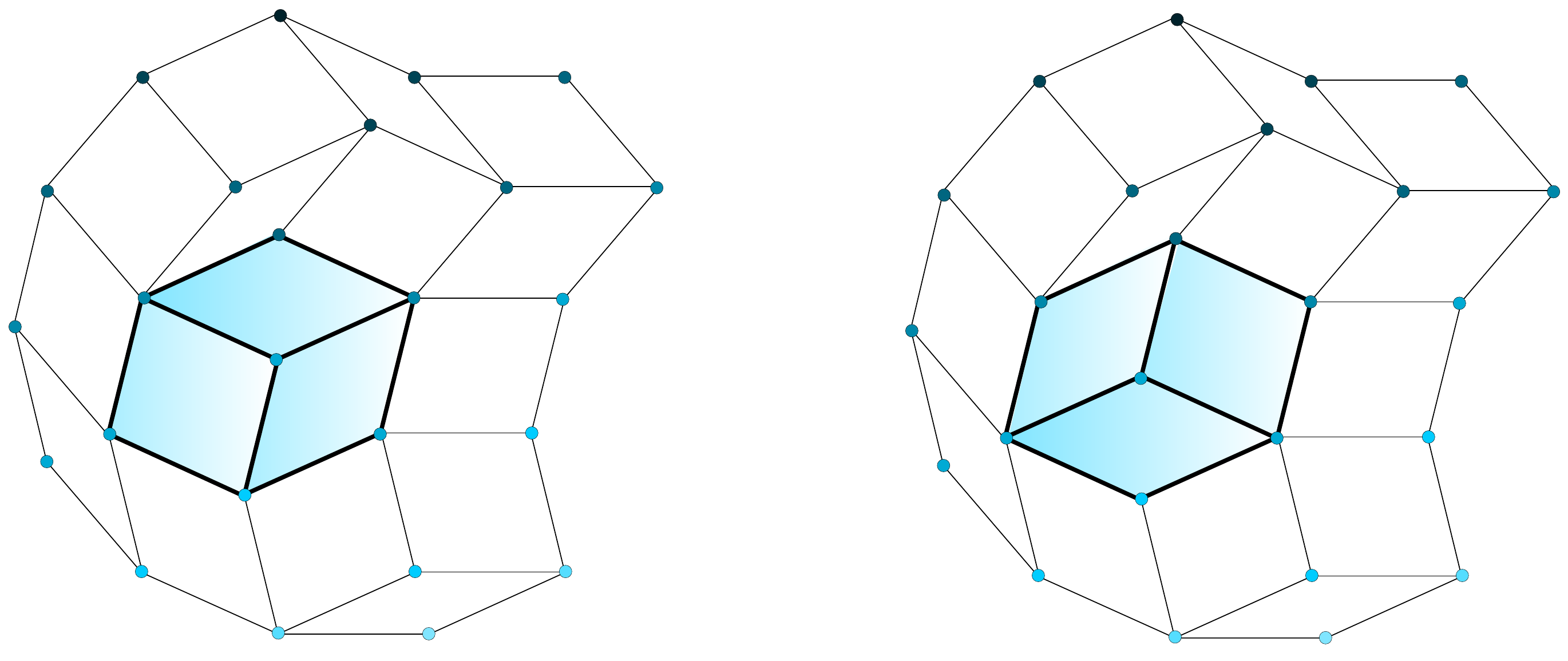}
\put(-170,60){$\longleftrightarrow$}
\end{center}
\caption{Two elements of ${\mathcal T}(7456312)$, related by a hexagon flip.}\label{fig:hex flip}
\end{figure}

This flip has a geometric interpretation. Let $T, T' \in \mathcal{T}(w)$ be two rhombic tilings that differ by a single flip. Let $T_H$ be the tiling of ${\sf E}(w)$ obtained from $T$ (or, equivalently, from $T'$) by erasing the three internal edges by which $T$ and $T'$ differ, and placing a hexagonal tile  in
the flip location. 
 As before, associate vector spaces $V_x$ to each vertex $x$ in $T_H$, where $\dim(V_x)$ equals the distance from $x$ to $\mathbb{C}^0$. The resulting space ${\mathcal Z}_{T_H}$ is similar to a Bott-Samelson variety: instead of being $\ell(w)$-fold iterated ${\mathbb C}{\mathbb P}^1$-bundles over the base flag, we replace three of these ${\mathbb C}{\mathbb P}^1$-bundles (corresponding to either triple of rhombi in the hexagon) by a ${\sf Flags}({\mathbb C}^3)$-bundle. We then have
\begin{center}
\includegraphics[scale=1.15]{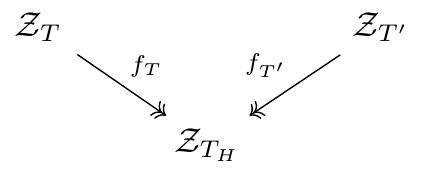}
\excise{
\begin{tikzcd} 
{\mathcal Z}_{T} \arrow[two heads]{rd}{f_T} &\phantom{2} &{\mathcal Z}_{T'} \arrow[two heads, ld, "f_{T'}"']\\
&{\mathcal Z}_{T_H}
\end{tikzcd}}
\end{center}
where the two maps are the projections determined by forgetting the vector space attached to the internal vertex of the hexagon. 

\subsection{Zonotopal tilings}\label{sec:zonotope}
The tiling $T_H$ described above is a special case of the ``zonotopal'' tilings of Elnitsky polygons, which were studied by the third author in \cite{Tenner-RDPP}. To be precise, a {\bf $\bm{2}$-zonotope} is the projection of a regular $q$-dimensional cube onto the ($2$-dimensional) plane; equivalently, a $2$-zonotope is a centrally symmetric convex polygon. A {\bf zonotopal tiling} of a region is a tiling by $2$-zonotopes. Figure~\ref{fig:zonotope} shows a zonotopal tiling
of ${\sf E}(87465312)$ using one octagon, three hexagons, and ten rhombi.

Let $\mathcal{T}_{zono}(w)$ be the collection of zonotopal tilings of ${\sf E}(w)$, in which the tiles ($2$-zonotopes) have sides of length one and edges parallel to edges of ${\sf E}(w)$. Because rhombi are a type of $2$-zonotope, we have $\mathcal{T}(w) \subseteq \mathcal{T}_{zono}(w)$.

\begin{figure}[htbp]
\begin{center}
\includegraphics[scale=0.28]{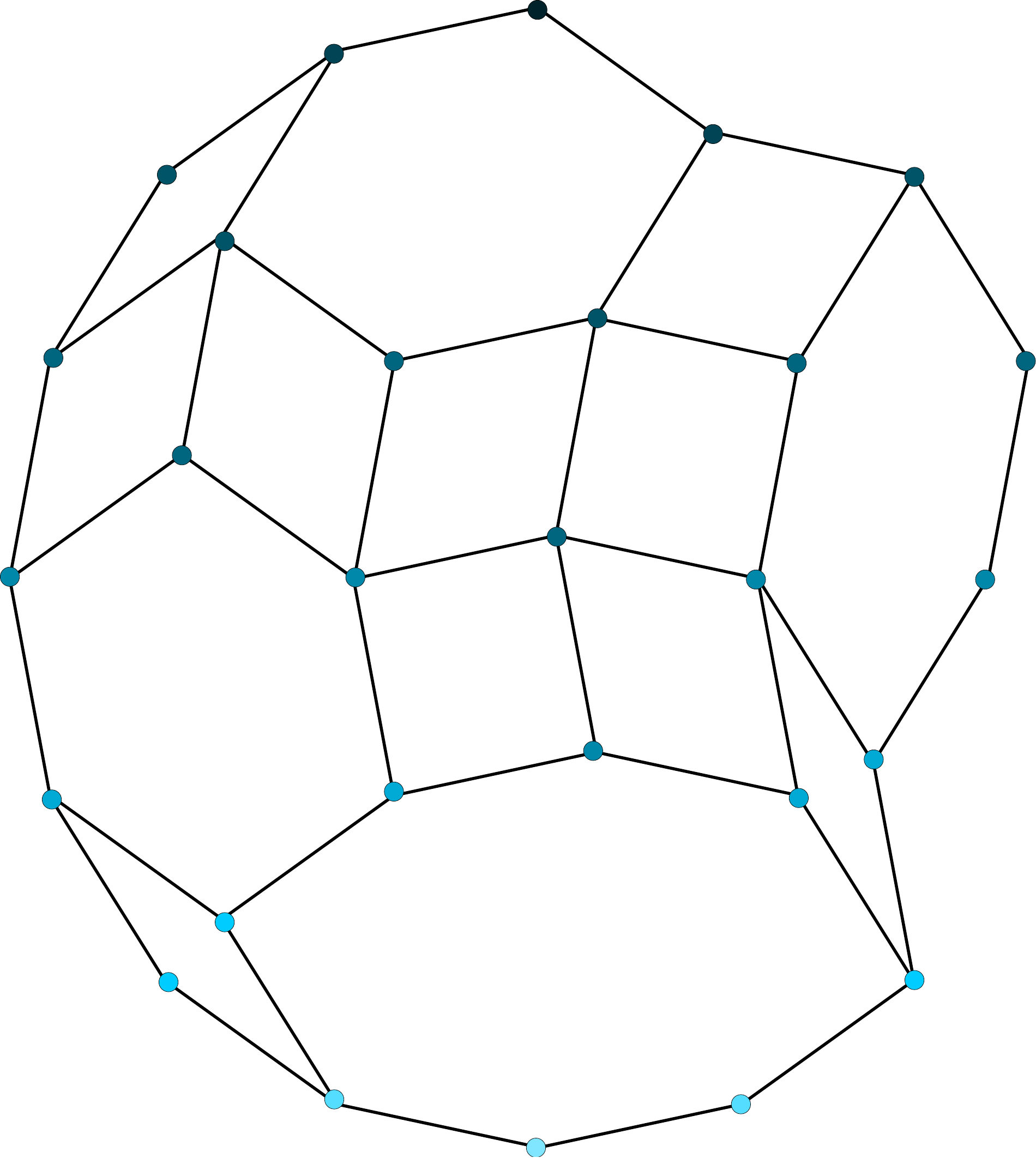}
\put(-79,-12){$\mathbb{C}^0$}
\put(-113,-5){$\mathbb{C}^1$}
\put(-144,17){$\mathbb{C}^ 2$}
\put(-162,49){$\mathbb{C}^ 3$}
\put(-169,82){$\mathbb{C}^ 4$}
\put(-162,116){$\mathbb{C}^ 5$}
\put(-144,146){$\mathbb{C}^ 6$}
\put(-110,168){$\mathbb{C}^ 7$}
\put(-78,175){$\mathbb{C}^ 8$}
\put(-44,-3){$G_1$}
\put(-14,22){$G_2$}
\put(-20,55){$G_3$}
\put(-4,80){$G_4$}
\put(2,113){$G_5$}
\put(-22,150){$G_6$}
\put(-50,157){$G_7$}
\end{center}\caption{A zonotopal tiling for the permutation 87465312.}\label{fig:zonotope}
\end{figure}
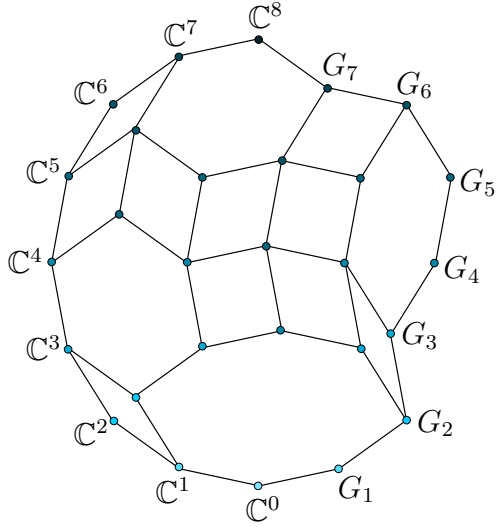

Given a zonotopal tiling $Z \in \mathcal{T}_{zono}(w)$, we can define its corresponding {\bf generalized Bott-Samelson variety} ${\mathcal Z}_{Z}$ by extending the construction from Section~\ref{sec:flip}. For each vertex $x$ in the zonotopal tiling, associate a vector space $V_x$ whose dimension is the minimal path length from the bottom vertex to $x$ along tile edges. Define 
\[{\mathcal Z}_{Z}:=\{(V_x: x\in {\tt Vert}(Z)): V_y\subseteq V_z \text{\ \ if $y\to z$}\}.
\] Let $T$ be a rhombic tiling that refines $Z$; ${\mathcal Z}_T$ may be constructed as iterated ${\mathbb C}{\mathbb P}^1$-bundles over a point.
In the analogous construction of ${\mathcal Z}_Z$, for each $2k$-gon of $Z$, we replace $k$ ${\mathbb C}{\mathbb P}^1$-bundles with a ${\sf Flags}({\mathbb C}^k)$-bundle. The variety ${\mathcal Z}_Z$ is smooth of dimension $\ell(w)$. Define $\pi_Z : {{\mathcal Z}}_Z \to X_w$
by forgetting all vector spaces except those labeled by the vertices $G_1,G_2,\ldots,G_{n-1}$. 

\begin{theorem}Given a zonotopal tiling $Z \in \mathcal{T}_{zono}(w)$, its corresponding generalized Bott-Samelson variety ${\mathcal Z}_{Z}$ together with the map $\pi_Z : {\mathcal Z}_{Z} \to X_w$ is a resolution of singularities.
\end{theorem}

\begin{proof}
Let $\pi_T : {\mathcal Z}_{T} \to X_w$ be a Bott-Samelson resolution where $T$ is any rhombic tiling that refines $Z$. We know that $\pi_T$ is birational, so let $\pi'_T$ be its rational inverse. Let $f : {\mathcal Z}_{T} \twoheadrightarrow {\mathcal Z}_{Z}$ be the projection determined by forgetting the vector spaces attached to the internal vertices of the $2k$-gons.
Since $f$ is surjective, the image of $\pi_Z$ is indeed $X_w$ and the following commutative diagram implies that $f\circ \pi'_T$ is a rational inverse to $\pi_Z$.
\begin{center}
\includegraphics[scale=1.15]{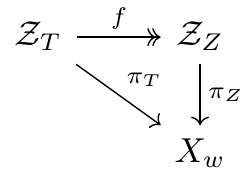}
\end{center}
\excise{
\[
\begin{tikzcd} 
{\mathcal Z}_{T} \arrow[two heads]{r}{f}\arrow{rd}{{\pi_T}}
                          &{\mathcal Z}_{Z} \arrow{d}{\pi_Z}\\
&X_w
\end{tikzcd}
\] } 
It follows that $\pi_Z : {\mathcal Z}_{Z} \to X_w$ is also a resolution of singularities. 
\end{proof}

The zonotopal tilings $\mathcal{T}_{zono}(w)$ of ${\sf E}(w)$ have a natural poset structure, as studied by the third author in \cite{Tenner-RDPP}. The order relation in this poset is given by reverse edge inclusion. Thus the rhombic tilings are the minimal elements in the poset. A pair of rhombic tilings differ by a single hexagon flip if and only if they are covered by a common element. Similarly, one can get a broader sense of how closely two rhombic tilings (equivalently, two commutation classes of reduced words for $w$) are related by determining their least upper bound in this poset. Geometrically, the relations in the poset $\mathcal{T}_{zono}(w)$ correspond to the projections ${\mathcal Z}_{Z} \twoheadrightarrow {\mathcal Z}_{Z'}$ between two generalized Bott-Samelsons for $X_w$.

By \cite[Theorem~6.13]{Tenner-RDPP}, the poset of zonotopal tilings of ${\sf E}(w)$ has a unique maximal element $\hat{Z}$ exactly in the case that $w$ avoids the patterns $4231$, $4312$, and $3421$. In this case, there is a distinguished  ${\mathcal Z}_{\hat{Z}}$ with a projection ${\mathcal Z}_{Z} \twoheadrightarrow {\mathcal Z}_{\hat{Z}}$ from every other generalized Bott-Samelson. Such permutations have been enumerated by T.~Mansour \cite{Mansour}.

For comparison, consider Elnitsky polygons whose zonotopal tilings do not contain any hexagonal tiles (equivalently, those polygons with a unique zonotopal tiling). These correspond to $321$-avoiding permutations, which are exactly those whose reduced words contain no long braid moves 
\cite[Theorem~2.1]{Billey-Jockusch-Stanley} (see also \cite[Section~3]{Tenner-RWM} for more general results relating pattern avoidance and reduced words). The unique tiling in this case is a deformation of the skew shape associated to the permutation by considering its {\it Rothe diagram} and removing empty rows and columns. 
A standard filling orders the tilings in the sense of \cite{Elnitsky} (and the final paragraph of the proof of Theorem~\ref{thm:main}).

We now have the following result (cf.~\cite[Remark~3.1]{Elek}, where this fact
for ordinary Bott-Samelsons is noted).

\begin{proposition}
Suppose that $Z\in \mathcal{T}_{zono}(w)$, and that the number of $2i$-sided tiles in $Z$ is $t_i$, for each $i\geq 1$. Then the Poincar\'e polynomial of the cohomology ring
$H^{\star}({\mathcal Z}_Z)$ is
\[\sum_{k=0}^{\ell(w)}\dim H^{2k}({\mathcal Z}_Z)q^k
=\prod_{i\geq 1}[i]_q!^{t_i},\] 
where $[i]_q:=1+q+q^2+\cdots + q^{i-1}$ and $[i]_q! := [i]_q  [i-1]_q \cdot \cdots  [1]_q$.
\end{proposition}
\begin{proof}
The variety ${\mathcal Z}_Z$ is constructed as iterated
flag bundles over a point, where $t_i$ of the fibrations are by ${\sf Flags}({\mathbb C}^{i})$.
It is a standard fact (following from the Schubert decomposition of ${\sf Flags}({\mathbb C}^{i})$)
that the Poincar\'{e} polynomial of 
$H^\star({\sf Flags}({\mathbb C}^{i}))$
is $[i]_q!$ (indeed, $[i]_q!$ is the ordinary generating function for ${\mathfrak S}_i$ with each permutation weighted by Coxeter length). The proposition now follows from the Leray-Hirsch theorem (cf.~\cite[Theorem~4D.1]{Hatcher}).
\end{proof}

\section{Additional discussion}\label{sec:final}
\begin{figure}[htb]
\begin{center}
\includegraphics[scale=0.45]{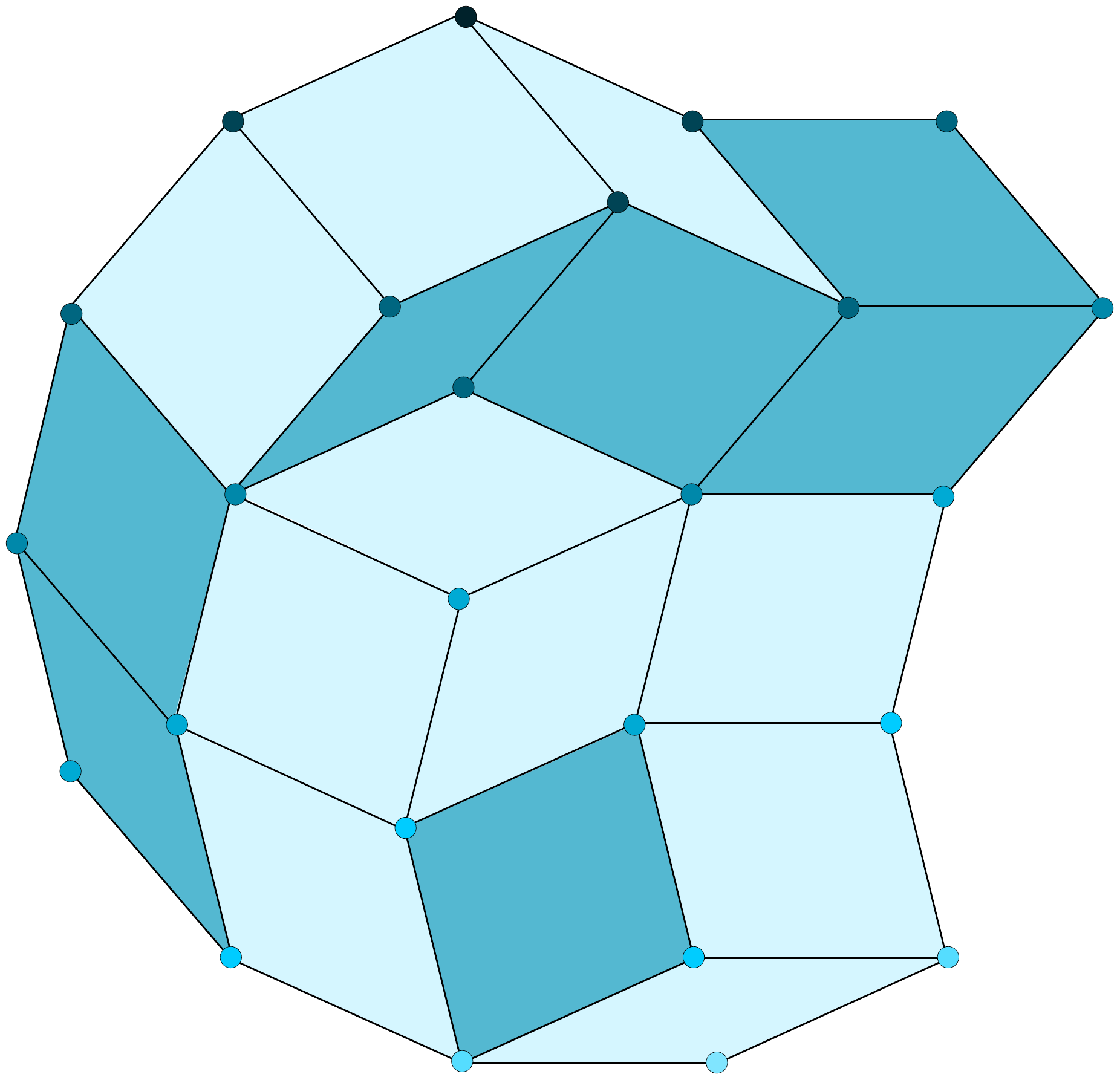}
\put(-90,-9){$\mathbb{C}^0$}
\put(-150,-8){$\mathbb{C}^1$}
\put(-210,18){$\mathbb{C}^ 2$}
\put(-245,60){$\mathbb{C}^ 3$}
\put(-255,110){$\mathbb{C}^ 4$}
\put(-240,168){$\mathbb{C}^ 5$}
\put(-200,211){$\mathbb{C}^ 6$}
\put(-148,234){$\mathbb{C}^ 7$}
\put(-200,79){$F_3^{(s_3)}$}
\put(-170,45){$\mathbb{C}^2$}
\put(-225,117){$F_4^{(s_3s_4)}$}
\put(-170,95){$F_3^{(s_3)}$}
\put(-174,162){$\mathbb{C}^5$}
\put(-126,186){$\mathbb{C}^6$}
\put(-135,148){$F_5^{(s_3s_4s_5)}$}
\put(-134,118){$F_4^{(s_3s_4)}$}
\put(-126,79){$F_3^{(s_3)}$}
\put(-90,32){$F_2^{(s_3s_2)}$}
\put(-33,22){$\mathbb{C}^1$}
\put(-46,75){$F_2^{(s_3s_2)}$}
\put(-35,122){$F_3^{(s_3)}$}
\put(-60,154){$\mathbb{C}^5$}
\put(0,163){$\mathbb{C}^4$}
\put(-96,212){$\mathbb{C}^6$}
\put(-42,212){$F_5^{(s_5)}$}\end{center}
\caption{A coloring corresponding to a fixed point of ${\mathcal Z}_{T}$.}\label{fig:colored tiling}
\end{figure}
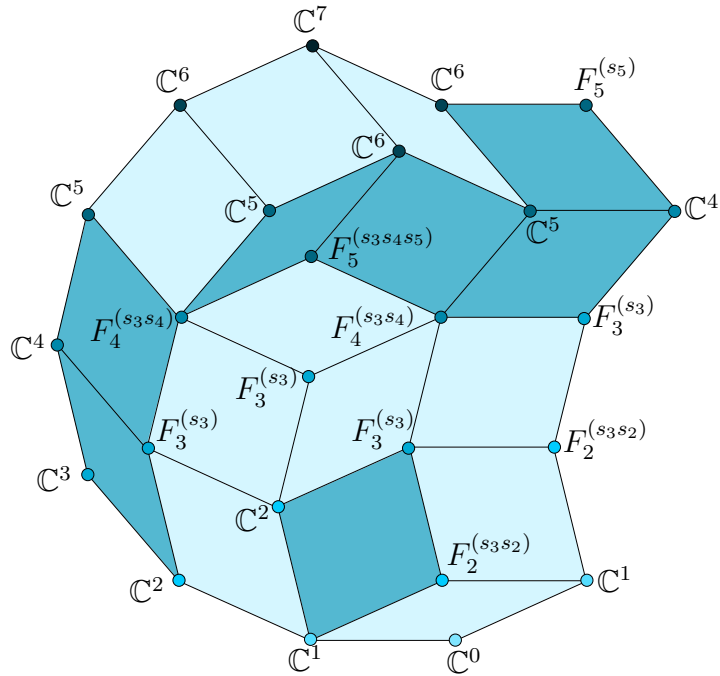

One may reformulate certain results about $BS^{\bf i}$ in terms of rhombic colorings; we refer to 
\cite[Section~3.2]{Escobar} for background with further references.

\begin{proposition} For $T\in\mathcal{T}(w)$, the $\sf{T}$-fixed points of ${\mathcal Z}_{T}$ (under the diagonal action) are in one-to-one correspondence 
with bipartitions of the rhombi of $T$.
\end{proposition}
\begin{proof}
Consider a $2$-coloring of the rhombi of $T$ representing the bipartition (as shown in Figure~\ref{fig:colored tiling}). There is a unique
way to choose $\{V_x\}_{x \in {\tt Vert}(T)}$ such that 
\begin{enumerate}
\item each $V_x$ is the span of a subset 
of the standard basis; and,
\item for any rhombus, its two vector spaces of common dimension are
the same (resp., different) if the rhombus is light-colored (resp., dark-colored).
\end{enumerate}
Since the $\sf{T}$-action is diagonal, if $\{V_x\}_{x \in {\tt Vert}(T)}$ is a ${\sf T}$-fixed point of ${\mathcal Z}_{T}$, then each $V_x$ must be $\sf{T}$-fixed, i.e., each $V_x$ must be spanned by a subset of the standard basis $\{e_1,\ldots,e_n\}$. Using the required containment relations, we can inductively determine $V_x$ for each vertex of $T$ by following an ordering of the rhombi given by a representative of the commutation class of $T$. At a particular colored rhombus, we make the two vector spaces of common dimension
the same (resp., different) if the rhombus is light-colored (resp., dark-colored).
\begin{center}
 \begin{tikzpicture}
    \node (top) at (0,0) {$V_c=V_a\bigoplus\langle e_b,e_c \rangle$};
    \node [below of=top] (nothing)  {$=$ (resp., $\neq$)};
    \node[left of=nothing] (new) {};
        \node[right of=nothing] (rnothing) {};

    \node[left of=new] (newnew) {};    
    \node [left of=newnew] (left) {$V_b=V_a\bigoplus\langle e_b\rangle$};
    \node [right of=rnothing] (right) {$V_x$};
        \node [below of=nothing] (down) {$V_a$};
\draw [thick,shorten <= +4pt] (newnew) -- (top);
\draw [thick,shorten <=-2pt] (top) -- (right);
\draw [thick,shorten <=-2pt] (right) -- (down);
\draw [thick,shorten <=+4pt] (newnew) -- (down);
\end{tikzpicture}
\end{center}
Conversely, every ${\sf T}$-fixed point can be indicated by such a coloring. 
\end{proof}

M.~Demazure \cite{Demazure} used the ${\sf T}$-fixed points to prove that the image of $BS^{(i_1, i_2, \ldots)}$ under the Bott-Samelson map $\pi$ is indeed
the Schubert variety $X_{s_{i_1}s_{i_2}\ldots}$. 
These fixed points are also useful in the study of moment polytopes of Bott-Samelson varieties, 
and for other applications.

These colorings also correspond to a stratification of ${\mathcal Z}_{T}$ by smaller Bott-Samelsons. Given a coloring, the corresponding stratum has the property that, for any light-colored rhombus, its two vector spaces of common dimension are equal. The dark-colored rhombi impose no conditions. The unique smallest stratum corresponds to the all-light coloring, whereas the unique largest stratum corresponds to the all-dark one. For background, see \cite[Section~4.3]{Escobar}

In \cite{Elnitsky}, the author extends his main construction to the other Weyl groups of classical Lie type. This seems related 
to the Bott-Samelsons for the associated Lie groups.

\section*{Acknowledgements}

We thank Allen Knutson and Alexander Woo for helpful comments. OP was supported by an NSF Graduate Research Fellowship. BT was partially supported by a Simons Foundation Collaboration Grant for Mathematicians. AY was supported by an NSF grant.

\bibliography{ElnBottSamarxiv}
\bibliographystyle{alpha}

\end{document}